\documentclass[12pt, reqno]{amsart}
\usepackage{amsmath, amstext, amsbsy, amssymb, amscd}
\usepackage[mathscr]{eucal}

  \usepackage[a4paper, margin=2.5cm]{geometry}

  \input xypic
  \xyoption{all}

 \usepackage{times}

  \usepackage{amsmath, amsthm, amsfonts,extarrows}

\newtheorem{theorem}{Theorem}[section]
\newtheorem{lemma}[theorem]{Lemma}

\theoremstyle{definition}
\newtheorem{definition}[theorem]{Definition}
\newtheorem{example}[theorem]{Example}

\newtheorem{proposition}[theorem]{Proposition}
\newtheorem{corollary}[theorem]{Corollary}

\theoremstyle{remark}
\newtheorem{remark}[theorem]{Remark}

\numberwithin{equation}{section}

  \renewcommand{\cH}{{\mathcal H}}

  \newcommand{\cE}{{\mathcal E}}
  
   \newcommand{\cN}{{\mathcal N}}

  \newcommand{\cU}{{\mathcal U}}
   \newcommand{\cV}{{\mathcal V}}
    \newcommand{\cW}{{\mathcal W}}
  \newcommand{\cG}{{\mathcal G }}
  \newcommand{\cC}{{\mathcal C }}

\newcommand{\cT}{{\mathcal T }}

\newcommand{\XX}{{\mathfrak X}}

   \newcommand{\ba}{\begin{eqnarray}}
   \newcommand{\na}{\end{eqnarray}}
   \newcommand{\ban}{\begin{eqnarray*}}
   \newcommand{\nan}{\end{eqnarray*}}

  \newcommand{\C}{{\mathbb C}}
  \newcommand{\R}{{\mathbb R}}
  \newcommand{\Z}{{\mathbb Z}}
  \newcommand{\Q}{{\mathbb Q}}

\begin{document}

\title[Gromov-Witten invariant of Weighted Blow-up]{Orbifold Gromov-Witten invariants of Weighted Blow-up\\ at Smooth Points}



  \author[Weiqiang He]{Weiqiang He}
  \address{Department of Mathematics\\  Sun Yat-Sen University\\
                        Guangzhou,  510275\\ China }
  \email{btjim1987@gmail.com}

  \author[Jianxun Hu] {Jianxun Hu$^1$}
  \address{Department of Mathematics\\  Sun Yat-Sen University\\
                        Guangzhou,  510275\\ China }
  \email{stsjxhu@mail.sysu.edu.cn}
\thanks{${}^1$Partially supported by NSFC Grant 11228101}


\date{}

\dedicatory{}

\begin{abstract}
 In this paper, one considers the change of orbifold Gromov-Witten invariants under weighted blow-up at smooth points. Some blow-up formula for Gromov-Witten invariants of symplectic orbifolds is proved. These results extend the results of manifolds case to orbifold case.

\end{abstract}

\maketitle

\tableofcontents

\section{Introduction}
The theory of Gromov-Witten invariant or quantum cohomology is probably one of the most important theories in mathematical physics
(especially in the string theory), and it has many applications in algebraic geometry and symplectic geometry.
Roughly speaking, given a symplectic manifold ($M$, $\omega$), Gromov-Witten invariant of $M$ is the number of pseudo-holomophic curves intersecting with some fixed homology classes of $M$.
There have been a great deal of activities to establish the mathematical foundation of the theory of quantum cohomology or Gromov-Witten invariants.
Y. Ruan and G. Tian \cite{RT} first established for semi-positive symplectic manifolds.
Then semi-positivity condition has been removed by many authors such as Li-Tian\cite{LT}, Fukaya-Ono\cite{FO}, Ruan\cite{R} and so forth.
In 2001, Li-Ruan \cite{LR} defined the relative Gromov-Witten invariants and established the degeneration formula.
Via this formula, we calculate Gromov-Witten invariants of $M$ when it can be `symplecticly' cut into two symplectic manifolds.
On the basis of degeneration formula, The second named author \cite{Hu} set up a blow-up type formula of Gromov-Witten invariants,
which tells the relations between Gromov-Witten invariants of a symplectic manifold M and some special invariants of its blow-ups at a smooth point or along a smooth curve.

Orbifolds, which were firstly introduced by I. Sataki \cite{Sat} in 1956, are a kind of generalization of manifolds.
Roughly speaking, an orbifold is a manifold equipping with some local group action.
During last few years, symplectic geometers pay more and more attentions to the category of orbifold.
They worked out that many symplectic surgeries of manifolds (such as symplectic cutting, symplectic gluing, blowing up and flops) can be generalized to orbifolds.
Moreover, numerous new characteristic emerges in the orbifold category because of the local group action.
Chen-Ruan \cite{CR1} established a new cohomology theory called Chen-Ruan cohomology, which is a good generalization of ordinary cohomology.
In 2000, Chen-Ruan \cite{CR2} generalized the quantum cohomology theory to orbifold and established the orbifold Gromov-Witten theory.
In 2010, B. Chen together with his collaborators \cite{CLSZ} defined the relative orbifold Gromov-Witten invariants and generalized the degeneration formula to the category of orbifolds.
In algebraic geometry, Abramovich and Fantechi \cite{AF} also obtained a similar degenertion formula.

In this paper, we will follow \cite{Hu} to study the change of orbifold Gromov-Witten invariants under weighted blow-up at smooth points. We will construct weighted blow-up in terms of symplectic cutting as in \cite{G} and use the degeneration formula to extend some blow-up formula of \cite{Hu} to the orbifold case.

Throughout this paper, let $\mathcal{G}$ be a compact symplectic orbifold (groupoid) of dimension 2n,
$\widetilde{\mathcal{G}}$ be the weighted blow-up of $\mathcal{G}$ at a smooth point.
Denote by $p:\widetilde{\mathcal{G}}\rightarrow\mathcal{G}$ the natural projection (cf. Remark 2.13 (3)).
Denote by $\Psi^{\mathcal{G}}_{(A,g,m,(\mathbf{g}))}(\alpha_1,\ldots,\alpha_m)$ the genus $g$ Gromov-Witten invariants of $\mathcal{G}$ with degree $A$,
$\Psi^{\mathcal G}_{(A,m,(\mathbf{g}))}(\alpha_1,\ldots, \alpha_m)$ the genus 0 GW invariants of $\mathcal G$ with degree $A$.
In this paper, we establish some relations between Gromov-Witten invariants of $\mathcal{G}$ and its blow-up $\widetilde{\mathcal{G}}$.
More precisely, we showed

\begin{theorem}\label{thm1-1}
Suppose that $\mathcal {G}$ is a compact symplectic orbifold of dimension $2n$ and $p: \widetilde{\mathcal G}\longrightarrow {\mathcal G}$ is the weighted blow-up of $\mathcal G$ at smooth point. $\alpha_i\in H^*_{CR}(\mathcal G)$, $i=1,2,\cdots, m$. Then for genus $g\leq 1$, $n\geq 2$, we have
$$
      \Psi^{\mathcal{G}}_{(A,g,m,(\mathbf{g}))}(\alpha_1,\ldots,\alpha_m)=\Psi^{\widetilde{\mathcal{G}}}_{(p^!(A),g,m,(\mathbf{g}))}(p^*\alpha_1,\ldots,p^*\alpha_m),
$$
where $p^!(A)= PDp^*PD(A)$, $PD$ stands for the Poincare dual.
\end{theorem}

If the (real) dimension of $\mathcal{G}$ is 4 or 6, then we can remove the genus condition and prove
\begin{theorem}\label{thm1-2}
   Under the assumption of Theorem \ref{thm1-1}.    If $1\leq n \leq 3$, then for any genus $g$, we have
$$
      \Psi^{\mathcal{G}}_{(A,g,m,(\mathbf{g}))}(\alpha_1,\ldots,\alpha_m)=\Psi^{\widetilde{\mathcal{G}}}_{(p^!(A),g,m,(\mathbf{g}))}(p^*\alpha_1,\ldots,p^*\alpha_m),
$$
where $p^!(A)=PDp^*PD(A)$, $PD$ stands for the Poincare dual.
\end{theorem}

    \section{Preliminaries  }
 In this section, we will briefly review the notions of orbifold in terms of orbifold atlas and proper
 \'etale groupoids. Then we recall the definition of ordinary cohomology of orbifold and the Chen-Ruan
 cohomology. Next we will focus on a concrete example of orbifold, the weighted projective space and show how to construct weighted blow-up.
Finally we will introduce the Gromov-Witten theory of orbifold and the degeneration formula.
The main references for this section are \cite{ALR, CLSZ, CR1, CR2, G, Sat}.

\subsection{Orbifolds and orbifold groupoids} 

  Let $X$ be a  paracompact Hausdorff space.
An $n$-dimensional orbifold chart for an open subset $U$ of $X$ is a triple $(\tilde U, G, \pi)$ given by a connected open subset  $\tilde U \subset \R^n$,
together  with an effective smooth action of a finite group $G$ such that   $\pi:  \tilde U\to    U  $ is the induced quotient map.
 An embedding  of orbifold charts
 \[
 \phi_{ij}:   (\tilde U_i, G_i, \pi_i)  \hookrightarrow   (\tilde U_j, G_j, \pi_j)
 \]
   is a smooth embedding
$
 \phi_{ij}:    \tilde U_i   \hookrightarrow    \tilde U_j,
$ covering the inclusion $\iota_{ij}: U_i \hookrightarrow  U_j$.
 As shown in \cite{MoePr}, such an embedding induces an injective group
 homomorphism  $\lambda_{ij}:  G_i \to G_j$ such that $\phi_{ij}$ is $G_i$-equivariant in the sense that
 \[
 \phi_{ij}(x\cdot g) =   \phi_{ij}(x )  \cdot   \lambda_{ij}( g),
 \]
 for $x\in \tilde U_i$ and $g\in G_i$.

\begin{definition}   An orbifold  atlas  on  $X$  is    a collection of   orbifold charts   $\cU =\{(\tilde U_i, G_i, \phi_i, U_i)\}$  for an open covering  $\{U_i\} $  of $X$ such that
\begin{enumerate}
\item  $\{ U_i\}$ is closed under finite intersection.
\item  Given any inclusion $U_i \subset  U_j$, there is an embedding of orbifold charts
$\phi_{ij}:   (\tilde U_i, G_i, \pi_i, U_i)  \hookrightarrow   (\tilde U_j, G_j, \pi_j, U_j)$.
 \end{enumerate}
Two  orbifold  atlases $\cU$ and $\cV$ are equivalent if there is a common orbifold atlas $\cW$  refining  $\cU$ and  $\cV$.
An (effective) orbifold $\XX = (X, \cU)$   is a paracompact Hausdorff space $X$  with an equivalence class of  orbifold  atlases or a maximal orbifold atlas.   Given an orbifold  $\XX = (X, \cU)$ and a point $x\in X$, let $(\tilde U, G, \pi)$ be an orbifold chart around $x$, then the local group at $x$ is defined to   be the
stabilizer of  $\tilde x \in \pi^{-1}(x)$, uniquely defined up to conjugation.
\end{definition}

There is also the conception of orbibunlde, which is the generalization of vector bundle:

\begin{definition}\label{orbibundle}
   Given a uniformized topological space $X$ and a topological space $E$ with a surjective continuous map $pr:E\longrightarrow X$, an orbifold structure of rank $k$ for E over U consists the following data:
   \begin{enumerate}
   \item An orbifold atlas $(V,G,\pi)$ of $X$.
   \item A uniformizing system $(V\times\mathbb{C}^k,G,\widetilde{\pi})$ for $E$. The action of $G$ on $V\times\mathbb{C}^k$ is an extension of the action of $G$ on $V$ given by $g(x,v)=(gx,\rho(x,g)v)$, where $\rho:V\times G\longrightarrow Aut(\mathbb{C}^k)$ is a smooth map satisfying :
   $$\rho(gx,h)\circ \rho(x,g)=\rho(x,h\circ g),\quad g,h\in G,x\in V$$
   \item The natural projection map $\widetilde{pr}:V\times\mathbb{C}^k\longrightarrow V$ satisfies $\pi \circ \widetilde{pr}=pr\circ \widetilde{\pi}$.
   \end{enumerate}
\end{definition}

 In particular, for a complex orbibundle $\mathcal{E}$ of rank 1 over the groupoid $\mathcal{G}$, we have

\begin{lemma}\label{lemma}
  Suppose $(\mathbb{C}^n\times\mathbb{C},G,\widetilde{\pi})$ is a uniformizing system for $\mathcal{E}$, then:

  (1)The map $\rho:\mathbb{C}^n\times G\longrightarrow Aut(\mathbb{C})(=\mathbb{C}^*)$ can remove the first variable, i.e.
$$
       \rho(x,g)=\rho(0,g),\quad \forall x\in \mathbb{C}^n, \forall g\in G.
$$

   (2)The action matrix of G over $\mathbb{C}^n\times\mathbb{C}$ is of the form:
$$
    \left(\begin{array}{cc}
             * & 0 \\
             0 & a(g)
    \end{array}\right).
$$
\end{lemma}

  The proof of the lemma is straightforward. Fix $g\in G$, observe that the eigenvalue of $\rho(x,g)$ is the same when $x$ change
  (since the eigenvalue is ``discrete" data, but the change is continuous.)
  Because $rank\mathcal{E}=1$, then the only element of the matrix $\rho(x,g)$ is just its  eigenvalue,
  so we get the first part, then the second part is trivial.

  \begin{remark} Using the language of groupoids, \cite{ALR} generalizes the definition of orbifold structure to noneffective group action. In this paper, we will follow their notation, use $\cG$ and $|\cG|$ to denote a general orbifold and its underlying topological space, $\cE$ to denote orbibundle. So we have $|\cG|=X$, $|\cE|=E$.
  \end{remark}

\begin{example}\label{ex}

  Define the action of the multiplicative group $\mathbb{C}^*$ on $\mathbb{C}^{n+1}-\{0\}$ by
$$
        \lambda\cdot (z_0,\ldots,z_n)=(\lambda^{m_0}z_0,\ldots,\lambda^{m_n}z_n),
$$
where the $m_i$ are integers greater than or equal to one. The quotient
  $$W\mathbb{P}(m_0,\ldots,m_n)=\{\mathbb{C}^{n+1}-\{0\}\}/\mathbb{C}^*$$
  is called a weighted projective space. We will give one of its orbifold atlas in section \ref{wps}.

\end{example}

Next we give the definition of orbifold morphisms or orbifold maps:

\begin{definition}
An orbifold morphism $f:\cG\longrightarrow\cH$ is a given orbifold atlas $\{(\tilde U_i, G_i, \phi_i, U_i)\}$ of $\cG$ and $\{(\tilde V_i, H_i, \psi_i, V_i)\}$ of $\cH$, and  an assignment of smooth maps $\{\tilde f_i:\tilde U_i\rightarrow\tilde V_i\}$ such that for any $g\in G_i$, there is $h\in H_i$ so that $h\cdot \tilde f_i(x)=\tilde f_i(g\cdot x)$ for any $x\in \tilde U_i$. And if $U_i\cap U_j\neq\varnothing$, $f_i$ and $f_j$ are compatible with respect to the orbifold structure of $U_i\cap U_j$. (cf. \cite{CR1} Def 2.1). An orbifold morphism $f$ induces a continuous map $|f|:|\cG|\longrightarrow|\cH|$.

\end{definition}

\begin{remark}\label{orc}
  There is also a equivalence relation between orbifold morphisms, which is called $\mathcal{R}$-equivalence.
  If two morphism are $\mathcal{R}$-equivalence, then we can treat them the same (See p.48 \cite{ALR} for more details).

\end{remark}

\subsection{Orbifold  cohomology} \

  Let $\cG$ be an  orbifold .
 If $|\cG|$ is compact,   the de Rham cohomology     of an   orbifold $\cG$,  denoted by
   $H^*_{orb}(\cG)$,
 is defined to be the cohomology of the $\cG$-invariant de Rham complex
   $(\Omega^p(\cG), d)$. Roughly speaking, given an orbifold atlas $\{(\tilde U_i, G_i, \phi_i, U_i)\}$ of $\cG$. An element in $\Omega^p(\cG)$ is an assignment $\{\omega_i \in \Omega^p(\tilde U_i)\}$ such that $g^*(\omega_i)=\omega_i$ for $\forall g \in G_i$, and $\phi_{i*}( \omega_i)|_{U_i\cap U_j}=\phi_{j*}( \omega_j)|_{U_i\cap U_j}$. The differential operator $d$ is induced by the differential operators of $\Omega^p(\tilde U_i)$. (See \cite{ALR} for details).

 The Satake's de Rham theorem \cite{Sat} for  an orbifold  $\cG$ leads to an isomorphism
$$
          H^*_{orb}(\cG) \cong H^*(|\cG|, \R)
$$
between the  de Rham cohomology and the singular cohomology of the underlying topological space.
Since in this paper we just consider the case $R=\R$, we treat them the same, and denote the cohomology group of $\cG$ by $H^*(\cG)$.

\begin{remark}
  In the case that $\cG$ is effective and $\cE$ is a complex line bundle, we can construct, via Chern-Weil theory, a Chern class $cw_1(\cE)$ in the de Rham cohomology group $H^2(\cG)$. (\cite{CR1})
\end{remark}

\subsection{Weighted projective spaces }\label{wps}

In this subsection, we recall some basic properties of weighted projective spaces and orbibundles over them.
For more detail, see \cite{G}.

The definition of weighted projective spaces $W\mathbb{P}(\bf m)$ is described in Example \ref{ex}(2). Now we give a natural orbifold atlas on it.
In fact, as is usually done for projective sapces, we can consider the sets
$$
         V_i=\{[{\bf z}]_{{\bf m}}\in W\mathbb{P}({\bf m})|z_i\neq 0\}\subset W\mathbb{P}(\bf m)
$$
and the bijective maps $\phi_i$ from $V_i$ to $\C^n/\mu_{m_i}({\bf \widehat{m}}^i)$ given by
$$
        \phi_i([{\bf z}]_{{\bf m}})=(\frac{z_0}{z_i^{\frac{m_0}{m_i}}},\ldots,\frac{\widehat{z_i}}{z_i},\ldots,\frac{z_n}{z_i^{\frac{m_n}{m_i}}})_{m_i},
$$
where $z_i^{\frac{1}{m_i}}$ is a $m_i$-root of $z_i$, $\mu_{m_i}$ is the set of $m_i$-roots of 1 and $(.)_{m_i}$ is a $\mu_{m_i}$-conjugacy class
in $\C^n/\mu_{m_i}({\bf \widehat{m}}^i)$ with $\mu_{m_i}$ acting on $\C^n$ by
$$
        \xi\cdot{\bf z}=(\xi^{q_0}z_0,\ldots,\xi^{q_n}z_n),\quad \xi\in\mu_{m_i}.
$$
Then on $\phi_i(V_j\cap V_i)\subset\C^n/\mu_{m_i}({\bf \widehat{m}}^i)$,
$$
       \phi_j\circ \phi_i^{-1}((z_0,\ldots,z_n)_{m_i})
       =(\frac{z_0}{z_j^{\frac{m_0}{m_j}}},\ldots,\frac{\widehat{z_j}}{z_j},\ldots,\frac{1}{z_j^{\frac{m_i}{m_j}}},\ldots,\frac{z_n}{z_j^{\frac{m_n}{m_j}}})_{m_j},
$$
so $W\mathbb{P}({\bf m})$ has the structure of an orbifold where all singularities have cyclic structure groups.
We can easily see, using symplectic reduction, that weighted projective spaces are symplectic orbifolds (Proposition 2.8 \cite{G}).

Consider the natural projection map
\begin{align*}
   \pi:W\mathbb{P}({\bf m}) & \longrightarrow  \mathbb{C}P^n \\
     [(z_0,\cdots,z_n)]  & \longmapsto   [z_0^{m_0}: \cdots :z_n^{m_n}].
\end{align*}
 $\pi$ is continuous and for all $ p\in \C P^n$, $\pi^{-1}(p)$ is a discrete point set. So we have
\begin{equation}\label{coh}
         H^*(W\mathbb{P}({\bf m}))=H^*(|W\mathbb{P}({\bf m})|,\R)=H^*(\C P^n,\R)=\R^{n+1}.
\end{equation}

In the case of usual projective spaces, any line bundles over $\mathbb{C}P^n$ is isomorphic to some line bundle $\mathcal{O}(k)$ of Chern class $k\in \Z$ of the form $\mathbb{S}^{2n+1}\times_{\mathbb{S}^1}\C$ with $\mathbb{S}^1$ acting on $\mathbb{S}^{2n+1}\times\C$ by
$$
       \lambda\cdot({\bf z},w)=(\lambda z_1,\ldots,\lambda z_n,\lambda^kw).
$$
Similarly, for any weighted projective space $W\mathbb{P}({\bf m})$, we can define the orbibundle $\mathcal{O}_{\bf m}(p/r)$ for any rational $p/r$ $(r>0)$, as the orbibundle $\mathbb{S}^{2n+1}\times_{\mathbb{S}^1}\C\rightarrow W\mathbb{P}(r{\bf m})$, with $\mathbb{S}^1$ action on $\mathbb{S}^{2n+1}\times\C$ given by
$$
        \lambda\cdot({\bf z},w)=(\lambda^{rm_1} z_1,\ldots,\lambda^{rm_n} z_n,\lambda^pw).
$$
There is a basic property for these line bundles (Remark 2.13 \cite{G}):

\begin{proposition}\label{norbund}
  The orbibundle $\mathcal{O}_{\bf m}(p/r)$ is isomorphic to the normal orbibundle of $W\mathbb{P}(r{\bf m})$ inside $W\mathbb{P}(r{\bf m},p)$
\end{proposition}

The orbibundle $\mathcal{O}_{\bf m}(p/r)$ can be considered as elements of the rational Picard group of $W\mathbb{P}({\bf m})$ and we have a natural identification
$$
        Pic(W\mathbb{P}({\bf m}))_{\Q}\cong H^2(W\mathbb{P}({\bf m}),\Q)\cong \Q
$$
by assigning to a line orbibundle its first Chern class. Moreover, we have the following result (Proposition 2.15 \cite{G}):

\begin{proposition}
  Every line bundle over $W\mathbb{P}({\bf m})$ is isomorphic to some $\mathcal{O}_{\bf m}(\chi)(\chi\in\Q)$ and its Chern class is given by
$$
      c_1(\mathcal{O}_{\bf m}(\chi))=\frac{\chi}{lcm(m_0,\ldots,m_n)}.
$$
\end{proposition}

\subsection{Weighted blow-up}\
In this subsection, we recall the construction of weighted blow-up. In this paper,  we will only discuss weighted blow-up at a smooth point.
See \cite{G} for more general case and  more details.

Suppose that $(\cG,\omega)$ is a symplectic orbifold, $H:\cG\rightarrow\R$ is a periodic hamiltonian function.
The hamiltonian vector field $X_H$ generates a circle action, which is compatible with the orbifold structure of $\cG$. Then $\cH=H^{-1}(0)$ is a suborbifold preserved by circle action.
Then we can obtain a symplectic orbifold $\cH/\mathbb{S}^1$ via symplectic reduction. More precisely, let
$$
       \pi : \cH \rightarrow \mathcal{Z}=\cH /\mathbb{S}^1.
$$
$\mathcal{Z}$ admits a natural symplectic structure $\tau_0$ such that
$$
            \pi^*\tau_0=i_0^*\omega,
$$
where $i^*_0:\cH=H^{-1}(0)\rightarrow \cG$ is the inclusion. Set
$$
           \cG^-=H^{-1}((0,+\infty)),\quad \cG^+=H^{-1}((-\infty,0)).
$$
Then after gluing $\mathcal{Z}$ with $\cG^+$ and $\cG^-$ respectively, we get
$$
             \overline{\cG}^+=\cG^+\bigsqcup\mathcal{Z},\quad \overline{\cG}^-=\cG^-\bigsqcup\mathcal{Z}.
$$
\cite{L} shows that there is a natural symplectic form on $\overline{\cG}^\pm$ such that it is $\omega$ when restricted on $\cG^\pm$ and is $\tau_0$ on $\mathcal{Z}$.
The surgery obtaining $\overline{\cG}^\pm$ from $\cG$ is called symplectic cutting. We also call $\overline{\cG}^\pm$ the symplectic cuts of $\cG$.

Recall that blow-up is just a special case of symplectic cutting.
Suppose that $p$ is a smooth point of $\cG$, then there is a Darboux uniformizing chart $(z_1, \ldots, z_n)$near $p$. Set the hamiltonian function as:
$$
         H({\bf z})=\Sigma^n_{i=0}|z_i|^2-k, \quad k>0.
$$
The induced $\mathbb{S}^1$-action is given by:
$$
        \lambda\cdot{\bf z}=(\lambda z_0,\ldots,\lambda z_n),\quad \lambda\in\mathbb{S}^1.
$$
Then we perform symplectic cutting on $\cG$ and get $\overline{\cG}^\pm$.
We observe that $\overline{\cG}^+\cong \C P^n$ and call $\widetilde{\mathcal{G}}:=\overline{\cG}^-$ the symplectic blow-up of $\cG$.
Roughly speaking, blow-up is obtained by cutting a disk near $p$ and collapsing the boundary via $\mathbb{S}^1$-action.

Similarly, if the hamiltonian function is:
$$
        H({\bf z})=\Sigma^n_{i=0}m_i\cdot|z_i|^2-k, \quad m_i\in \Z^+, \quad k>0.
$$
 Collapse the boundary $H^{-1}(0)$ with the induced $\mathbb{S}^1$-action:
$$
         \lambda\cdot{\bf z}=(\lambda^{m_0} z_0,\ldots,\lambda^{m_n} z_n),\quad \lambda\in\mathbb{S}^1,\quad m_i\in\Z^+.
$$
 After performing symplectic cutting, we get the weighted blow-up $\widetilde{\mathcal{G}}:=\overline{\cG}^-$, and ${\bf m}=(m_0,\ldots,m_n)$ is called its weight.

 \begin{remark}\begin{enumerate}
   \item We observe that $\overline{\cG}^+$ is $W\mathbb{P}(m_0,\cdots,m_n,1)$. This can be obtained by another description of symplectic cut given by Lerman \cite{L}.
   Consider the symplectic orbifold $(\cG\times\C,\omega\oplus\sqrt{-1}dw\wedge d\overline{w})$, set the Hamiltonian function as:
$$
         \mathbf{H}({\bf z},w)=H({\bf z})+|w|^2=\Sigma^n_{i=1}m_i\cdot|z_i|^2+|w|^2-k.
$$
   The induced circle action is $\lambda\cdot({\bf z,w})=(\lambda^{m_0} z_0,\cdots,\lambda^{m_n} z_n,\lambda w)$.
   Lerman \cite{L} shows that:
$$
                  \overline{\cG}^+=\{({\bf z},w)\in\cG\times\C|\mathbf{H}({\bf z},w)=0\}/\mathbb{S}^1.
$$
   Then we easily get $\overline{\cG}^+=W\mathbb{P}(m_1,\ldots,m_n,1)$ after a scaling.
   \item Since the exceptional divisor $\mathcal{Z}\cong W\mathbb{P}(m_0,\ldots,m_n)$, from Proposition 2.8 and (1)
   we know that the normal bundle of $\mathcal{Z}$ in $\overline{\cG}^\pm$ is $\mathcal{O}_{{\bf m}}(\pm1)$ respectively.
   \item Because the normal bundle of $\mathcal{Z}$ in $\widetilde{\cG}$ is $\mathcal{O}_{{\bf m}}(-1)$, from the view of topology, blow-up is
   removing $p$ and gluing $\mathcal{O}_{{\bf m}}(-1)$ on $\cG\backslash\{p\}$ via the projection:
\begin{align*}
     p:\mathcal{O}_{{\bf m}}(-1)=\mathbb{S}^{2n+1}\times_{\mathbb{S}^1}\C & \longrightarrow  \C^{n+1}   \\
          [({\bf z},w)] &\longmapsto (w^{m_0}z_0,\ldots,w^{m_n}z_n ).
\end{align*}
   Note that $p^{-1}(0)=W\mathbb{P}({\bf m})$, and $p^{-1}(\C^{n+1}-\{0\})\cong \C^{n+1}-\{0\}$, which is similar to the ordinary blow-up.
   The map $p$ also induces a natural projection map $p:|\widetilde{\cG}|\rightarrow|\cG|$.
   \end{enumerate}
 \end{remark}

\subsection {Orbifold Gromov-Witten theory and degeneration formula}\

In this subsection, we briefly recall the definition of (absolute) Gromov-Witten theory of an almost complex, compact orbifold $\cG$. Then we recall the definition of relative Gromov-Witten theory of a symplectic pair $(\cG,\mathcal{Z})$.
Finally we introduce the degeneration formula which expresses absolute invariants in terms of relative Gromov-Witten invariants. One can see the
original paper \cite{CR1, CR2, CLSZ} for details. But we need to introduce an important notion before:

Let $\cG  = (X, \cU) $ be an orbifold, then the set of pairs
$$
         \tilde X =\{ (x, (g)_{G_x})|  x\in X, g\in G_x\},
$$
where $(g)_{G_x}$ is the conjugacy class of $g$ in the local group $G_x$, has a natural
orbifold structure given by
$$
         \{ \bigl(\tilde U^{g}, Z_G(g),  \tilde\pi,  \tilde U^{g}/C(g) \big) |  g\in G\}.
$$
Here for each orbifold chart $( (\tilde U,  G,  \pi, U) \in \cU$,  $Z_G(g)$ is the centralizer of $g$ in $G$ and
$\tilde U^{g}$ is the fixed-point set of $g$ in $\tilde U$.  This orbifold, denoted by $\bigwedge\cG$, is called the inertia orbifold of $\cG$.
The inertia orbifold $\bigwedge \cG$  consists of disjoint union of sub-orbifolds of $\cG$. To describe the connected components of $\bigwedge \cG$,
we need to introduce an equivalence relation on the set of conjugacy classes in local groups as in \cite{CR1}.  For each $x\in X$, let
$(\tilde U_x,  G_x,  \pi_x, U_x)$ be a local orbifold chart at $x$. If $y\in U_x$, up to conjugation, there is an injective homomorphism of local groups $G_y\to G_x$, hence
the conjugacy class $(g)_{G_x}$ is well-defined for $g\in G_y$. We define the equivalence to be generated
by the relation $(g)_{G_y}\sim (g)_{G_x}$. Let $\cT_1$ be the set of equivalence classes, then
$$
        \bigwedge \cG = \bigsqcup_{(g) \in \cT_1} \cG_{(g)},
$$
where $\cG_{(g)} =\{(x, (g')_{G_x})| g'\in G_x, (g')_{G_x} \sim  (g)\}$.
 Note that
$\cG_{(1)} =\cG$ is called the non-twisted sector and  $\cG_{(g)}$ for $g\neq 1$ is called a twisted sector
of $\cG$.  Similarly, the  $k$-sectors $\bigwedge\cG^{[k]}$ of $\cG$ is defined to be the orbifold  on the
set of all pairs
$$
  (x, (g_1, \cdots, g_k)_{G_x}),
$$
where $(g_1, \cdots, g_k)_{G_x}$ denotes the conjugacy  class of k-tuples.  Here two
k-tuples $(g_1^{(i)}, \cdots, g_k^{(i)})_{G_x}$, $i=1, 2$, are conjugate if there is $g\in G_x$ such that
$g_j^{(2)} = g g_j^{(1)}g^{-1}$ for all $j=1, \cdots, k$.  The  $k$-sector orbifold $\bigwedge\cG^{[k]}$
consists of disjoint union of sub-orbifolds of $\cG$
$$
        \bigwedge \cG^{[k]} = \bigsqcup_{({\bf g}) \in \cT_k} \cG_{({\bf g})},
$$
where $\cT_k$ denotes the set of  equivalence classes of  conjugacy   k-tuples in local groups.

  The   degree shifting number $\iota: \bigsqcup_{(g) \in \cT_1} \cG_{(g)}   \to \Q$,   defined by Chen and Ruan in \cite{CR1},
   is   determined by the canonical  automorphism  $\Phi$ on the normal bundle   of $e=\sqcup e_{(g)}:  \sqcup \cG_{(g)}  \to \cG$
$$
          \cN_e  =\bigsqcup_{(g) \in \cT_1} \cN_{(g)}  = \bigsqcup_{(g) \in \cT_1}  e^*_{(g)} T\cG/T\cG_{(g)},
$$
  where the  automorphism  $\Phi$ acting on the normal bundle $\cN_{(g)} $ over $\cG_{(g)}$
   is given by the canonical $g$-action on the complex vector bundle  over the orbifold chart $\tilde U^g$ of
  $\cG$.
 Over each   connected component  $\cG_{(g)}$,  the normal bundle  $\cN_{(g)} $ has an  eigen-bundle decomposition
$$
    \cN_{(g)}  = \bigoplus_{\theta_{(g)} \in \Q\cap (0, 1) } \cN  (\theta_{(g)})
$$
  where $\Phi$ on $\cN   (\theta_{(g)})  $ is  the multiplication by $e^{2\pi \sqrt{-1}\theta_{(g)}}$.
   Then  the degree shifting number
$$
   \iota_{(g)}  = \sum_{\theta_{(g)}}  rank_\C (\cN (\theta_{(g)}) ) \theta_{(g)} ,
$$
  defines   a locally constant function on $\bigwedge\cG$.

\begin{lemma}(\cite{ALR})\label{dim}
    There exists a natural orbifold isomorphism $i:\cG_{(g)}\rightarrow \cG_{(g^{-1})}$. Moreover
$$
        dim \cG_{(g)}=dim \cG_{(g^{-1})}=dim \cG - \iota_{(g)}- \iota_{(g^{-1})}.
$$
\end{lemma}

Now we come back to the discussion of orbifold Gromov-Witten theory.
First of all, we consider the definition of orbifold Riemann surface or orbicurve.

\begin{definition}
  A (nodal) orbicurve $\cC$ is a nodal marked Riemann surface with an orbifold structure as follows:
  \begin{enumerate}
    \item The singular point set of each component is contained in the set of marked points and nodal points;
    \item A neighborhood of a singular mark point is covered by the orbifold chart $(\mathbb{D},\Z_r,\phi)$, where the $\Z_r$-action is given by
$$
         z\mapsto e^{2n\pi \sqrt{-1}/r}z,\quad e^{2n\pi \sqrt{-1}/r}\in\Z_r.
$$
    \item A singular nodal point must satisfy the balance condition, i.e one of its neighborhoods can be uniformized by the chart $(\widetilde{U},\Z_s,\psi)$,
    where $\widetilde{U}=\{(z,w)\in\C^2|zw=1\}$, the $\Z_s$-action is given by $(z,w)\mapsto (e^{2n\pi \sqrt{-1}/s}z,e^{-2n\pi \sqrt{-1}/s}w),\quad e^{2n\pi \sqrt{-1}/r}\in\Z_s$
  \end{enumerate}
\end{definition}

Suppose $(\cG,\omega)$ is a symplectic orbifold with a tamed complex structure $J$,
$\cC$ is an orbicurve , $f:\cC\rightarrow\cG$ is a $J$-holomorphic orbifold morphism.
If $x\in |\cC|$ is a singular point of $\cC$, $f$ maps $x$ to $y\in |\cG|$, and induces a homomorphism  between their local group $\lambda_f:\Z_r\rightarrow G_y$.
$f$ is called representable if $\lambda_f$ is injective for all singular point $x$. Similar to the manifold case, we can define:

\begin{definition}
 A stable orbifold morphism $f:\cC\rightarrow\cG$ is a representable ,$J$-holomorphic morphism from an orbicurve $\cC$ to $\cG$ with a finite automorphism.
 The moduli space $\overline{\mathcal{M}}_{g,m,A}(\mathcal{G})$ consists of all the equivalence class of stable orbifold morphism of genus $g$, $m$ marked points and degree $A\in H_2(|\cG|,\Z)$.
 ($f$ and $f'$ are said to be equivalent if $\exists \phi \in Aut(\cC)$ such that $f'=f\circ \phi$ up to an $R$-equivalence. )
\end{definition}

For each marked point $x_i$, there is an evaluation map
\begin{eqnarray*}
  ev_i:\overline{\mathcal{M}}_{g,m,A}(\mathcal{G}) &\rightarrow & \bigwedge\cG \\
          (\cC,f) & \mapsto & (y_i, (g_i)_{G_{y_i}})
\end{eqnarray*}
where $f$ maps $x_i$ to $y_i$, and $g_i=\lambda_f(\sigma)$. (Here $\sigma$ is the generator of the local group $\Z_i$ of $x_i$.)

We can use the decomposition of $\bigwedge \cG$ to decompose $\overline{\mathcal{M}}_{g,m,A}(\mathcal{G})$ into components:
$$
     \overline{\mathcal{M}}_{g,m,A}(\mathcal{G})=\bigsqcup_{(g_i)\in\cT_{\cG}}\overline{\mathcal{M}}_{g,m,A}(\mathcal{G})((g_1),\ldots,(g_m)),
$$
where $\overline{\mathcal{M}}_{g,m,A}(\mathcal{G})((g_1),\ldots,(g_m))$ is the component being mapped into $\cG_{(g_i)}$ by $ev_i$.
For simplicity, we set ${(\bf g})=((g_1),\ldots,(g_m))$, denote the component by $\overline{\mathcal{M}}_{g,({\bf g}),A}(\mathcal{G})$.

Chen-Ruan \cite{CR2} observed that each component of the moduli space has a virtual fundamental class of the expected dimension.

\begin{proposition}\label{abdim}
  The moduli space $\overline{\mathcal{M}}_{g,({\bf g}),A}(\mathcal{G})$  carries a virtual fundamental cycle $[\overline{\mathcal{M}}_{g,({\bf g}),A}(\mathcal{G})]^{vir}$ with the expected dimension
$$
    vdim \overline{\mathcal{M}}_{g,({\bf g}),A}(\mathcal{G})=c_1(A)+(n-3)(1-g)+m-\iota_{({\bf g})},
$$
where $\iota_{({\bf g})}=\Sigma_{i=1}^m \iota_{(g_i)}$ and $\iota_{(g_i)}$ is the degree shifting number for $\cG_{(g)}$.
\end{proposition}

Now we can define the orbifold Gromov-Witten invariants as:
$$
      \Psi^{\mathcal{G}}_{(A,g,m,(\mathbf{g}))}(\alpha_1,\ldots,\alpha_m)=\int_{[\overline{\mathcal{M}}_{g,({\bf g}),A}(\mathcal{G})]^{vir}}\prod_i ev^*_i(\alpha_i),
$$
where $\alpha_i\in H^*(\cG_{(g_i)})$.

Let $(\cG,\mathcal{Z})$ be a relative pair, which means that $\mathcal{Z}$ is a symplectic divisor of $\cG$, $\cN$ is its normal bundle.
Similar to manifold case, we will consider the moduli space of all the $J$-holomorphic maps $f:\cC\rightarrow\cG$ intersecting divisor $\mathcal{Z}$ in finite relative marked point with prescribed contact order.
If we fix a relative marked point $x$, $f(x)=y$, then there is an orbifold atlas $(\widetilde{U},\Z_r,\phi)$ near $x$, and $(\widetilde{V},G_y,\psi)$ near $y$, such that $f$ can be lifted to be a smooth map $\widetilde{f}:\widetilde{U}\rightarrow\widetilde{V}$.
The (fractional) contact order $\ell_x$ is defined to be $\frac{d}{r}$. Here $d$ is the contact order of $\widetilde{f}$ and $r$ is the order of the local group of $x$.

Note that the moduli space can be compactified via similar scheme to the manifold case. Denote $\mathcal{Q}:=\mathbb{P}(\cN\oplus\C)$, we can glue $m$ copies of $\mathcal{Q}$ together with identifying one's infinity section $\mathcal{Z}_{i,\infty}$ to another's zero section $\mathcal{Z}_{i+1,0}$.
Denote the result space by $\mathcal{Q}_m$, and $\mathcal{Z}_{i,\infty}=\mathcal{Z}_{i+1,0}$ by $\mathcal{Z}_i$.
Set $\mathcal{G}_m:=\cG\wedge_{\mathcal{Z}}\mathcal{Q}_m$. Then we have

\begin{definition}
  A stable relative orbifold holomorphic morphism $f:\cC\rightarrow\cG_m$ is a representable, $J$-holomorphic morphism satisfying:
  \begin{enumerate}
    \item The rigid components are mapped into $\cG$ and the rubber components are mapped into $\mathcal{Q}_m$.
    \item The relative marked points are mapped into $\mathcal{Z}_{m,\infty}$ and the sum of fractional contact orders equals to $\mathcal{Z}\cdot A$.
    \item The relative nodes are mapped into Sing$\cG_m$ satisfying balanced condition that the two branches at the node are mapped to different irreducible components of $\cG_m$ and the contact orders to $\mathcal{Z}_{i,0}=\mathcal{Z}_{i-1,\infty}$ are equal.
    \item The automorphism group is finite.
  \end{enumerate}
  $f$ and $f'$ are said to be equivalent if there exist $\phi\in Aut(\cC)$ and $\psi \in Aut(\mathcal{Q}_l)$, such that $f\circ\phi=\psi\circ f'$ up to an $R$-equivalence.
  The moduli space $\overline{\mathcal{M}}_{g,l,A,T_k}(\mathcal{G},\mathcal{Z})$ consists of all stable relative orbifold morphism with genus $g$, homologous class $A$, $l$ absolute mark points, $k$ relative marked points with the contact orders prescribed by $T_k=(\ell_1,\ldots,\ell_k)$.
\end{definition}

Similar to the case of manifold, For each absolute marked point $x_i$, we have an evaluation map:
$$
    ev_i:\overline{\mathcal{M}}_{g,l,A,T_k}(\mathcal{G},\mathcal{Z})\rightarrow \bigwedge\cG, \quad i=1, \cdots, l.
$$
For each relative marked point $y_i$, we have a relative evaluation map :
$$
     ev_j^r:\overline{\mathcal{M}}_{g,l,A,T_k}(\mathcal{G},\mathcal{Z})\rightarrow \bigwedge\mathcal{Z}.
$$
Let $({\bf g})=\{(g_1),\ldots,(g_l)\}$, $({\bf h})=\{(h_1),\ldots,(h_k)\}$. The decomposition of $\bigwedge\cG$ and$\bigwedge\mathcal{Z}$ induces a decomposition of the moduli space as follows
$$
       \overline{\mathcal{M}}_{g,l,A,T_k}(\mathcal{G},\mathcal{Z})=\bigsqcup_{({\bf g}),({\bf h})}\overline{\mathcal{M}}_{g,({\bf g}),A,({\bf h}),T_k}(\mathcal{G},\mathcal{Z}).
$$
Chen-Li-Sun-Zhao \cite{CLSZ}  and Abramovich-Fantechi \cite{AF} show the following proposition.

\begin{proposition}\label{redim}
   The moduli space $\overline{\mathcal{M}}_{g,({\bf g}),A,({\bf h}),T_k}(\mathcal{G},\mathcal{Z})$ carries a virtual fundamental class $[\overline{\mathcal{M}}_{g,({\bf g}),A,({\bf h}),T_k}(\mathcal{G},\mathcal{Z})]^{vir}$ with the expected dimension
\begin{equation}\label{redim-1}
    vdim\overline{\mathcal{M}}_{g,({\bf g}),A,({\bf h}),T_k}(\mathcal{G},\mathcal{Z})=c_1(A)+(3-n)(g-1)+m+k-\iota_{(\mathbf{g})}^\mathcal{G}-\iota_{(\mathbf{h})}^\mathcal{G}-\sum_i[\ell_i],
\end{equation}
where $dim_{\mathbb{R}}M=2n$, $\iota_{(\mathbf{g})}^\mathcal{G}=\sum_i\iota_{(g_i)}^\mathcal{G}$, $\iota_{(g_i)}^\mathcal{G}$ is the degree shifting number of the component $\cG_{(g_i)}$ of $\bigwedge\mathcal{G}$. $\iota_{(\mathbf{h})}^\mathcal{G}$ is defined similarly.  $[\ell_i]$ is the biggest integer less than the fractional contact order $\ell_i$.
\end{proposition}

The orbifold relative Gromov-Witten invariants are defined as
\begin{eqnarray*}
      & & \Psi^{(\mathcal{G},\mathcal{Z})}_{(A,g,(\mathbf{g}),({\bf h}),T_k})(\alpha_1,\ldots,\alpha_m|\beta_1,\ldots,\beta_k)\\
      & = & \frac{1}{|Aut(\mathcal{T}_k)|}\int_{[\overline{\mathcal{M}}_{g,({\bf g}),A,({\bf h}),T_k}
       (\mathcal{G},\mathcal{Z})]^{vir}}\prod_i ev^*_i(\alpha_i)\prod_jev^{r,*}_j(\beta_j),
\end{eqnarray*}
where $\alpha\in H^*(\cG_{(g_i)})$, $\beta_j\in H^*(\mathcal{Z}_{(h_j)})$, $\mathcal{T}_k=\{(\ell_1,(h_1),\beta_1),\ldots,(\ell_k,(h_k),\beta_k)\}.$

Li-Ruan \cite{LR} gave a degeneration formula  which expresses the absolute Gromoc-Witten invariants of a manifold $M$ in terms of the  relative Gromov-Wiiten invariants of its symplectic cuts.
Chen-Li-Sun-Zhao\cite{CLSZ} extended this degeneration formula to the orbifold case in the differential category. Abramovich-Fantechi \cite{AF} also obtained this formula in the case of algebraic stacks.

Suppose $(\cG,\omega)$ is a symplectic orbifold. After performing symplectic cutting on $\cG$ we obtain two symplectic orbifold $\overline{\cG}^\pm$.
One can glue two pseudoholomorphic curve $(u^+,u^-)$ in $\overline{\cG}^+$, $\overline{\cG}^-$ with the balance condition to obtain a pseudoholomorphic curve $u$ in $\cG$.
Now we have a projective map
$$
          \pi:\cG\rightarrow\overline{\cG}^+\wedge_{\mathcal{Z}}\overline{\cG}^-,
$$
where $\overline{\cG}^+\wedge_{\mathcal{Z}}\overline{\cG}^-$ is the orbifold obtained via gluing $\overline{\cG}^\pm$ along the divisor $\mathcal{Z}$.
$\pi$ induces a homomorphism
$$
          \pi_*:H_2(|\cG|,\Z)\rightarrow H_2(|\overline{\cG}^+\wedge_{\mathcal{Z}}\overline{\cG}^-|,\Z).
$$
Then $(u^+,u^-)$ defines a homology class $[u^++u^-]\in H_2(|\overline{\cG}^+\wedge_{\mathcal{Z}}\overline{\cG}^-|,\Z)$. Moreover, we have
$$
             [u^++u^-]=\pi_*([u]).
$$

Note that $\pi_*$ is not injective, and elements in $\ker \pi_*$ are called vanishing cycle (cf. \cite{LR}). Let $[A]:=A+\ker \pi_*$. Define
$$
              \Psi^{\mathcal{G}}_{([A],g,m,(\mathbf{g}))}(\alpha_1,\ldots,\alpha_m)=\sum_{B\in[A]}\Psi^{\mathcal{G}}_{(B,g,m,(\mathbf{g}))}(\alpha_1,\ldots,\alpha_m).
$$
By Gromov's compactness theorem, the summation of right hand side is finite.

While $\cG$ degenerates to $\overline{\cG}^+\wedge_{\mathcal{Z}}\overline{\cG}^-$ , the moduli space $\overline{\mathcal{M}}_{g,({\bf g}),[A]}(\mathcal{G})$ also degenerates to $\overline{\mathcal{M}}_{g,({\bf g}),\pi_*[A]}(\overline{\cG}^+\wedge_{\mathcal{Z}}\overline{\cG}^-)$,
which consists of the components indexed by the possible relative type $\Gamma$ of $u^\pm$.
Using the virtual neighborhood techniques \cite{R,CLW}, Chen-Li-sun-Zhao\cite{CLSZ} defined GW invariants $\Phi_\Gamma$ for each component indexed by $\Gamma$ and proved
$$
         \Psi^{\mathcal{G}}_{([A],g,m,(\mathbf{g}))}=\sum \Psi_\Gamma.
$$
For simplicity, we will assume that $u^\pm$ has just one component. Denote its index by
$$
      \Gamma=\{A^+,g^+, m^+, ({\bf g}^+),({\bf h}^+),T_k^+;A^-,g^-,m^-,({\bf g}^-),({\bf h}^-),T_k^-\}.
$$
satisfying
\begin{enumerate}
  \item $A^++A^-=\pi_*([A])$, $g=g^++g^-+k-1$, $({\bf g}^+)\cup({\bf g}^-)=({\bf g})$, $m^+ + m^- = m$,
  \item $({\bf h}^+)=(({\bf h}^-)^{-1})$, $T_k^+=T_k^-$.
\end{enumerate}
Then we have the following degeneration formula

\begin{theorem}
  Suppose that $\alpha^\pm_i\in H^*(\bigwedge \cG^\pm)$ with $\alpha^+_i|_{\bigwedge \mathcal{Z}}=\alpha^-_i|_{\bigwedge \mathcal{Z}}$ defines a class $(\alpha^+_i,\alpha^-_i)$ in $H^*(\bigwedge(\overline{\cG}^+\wedge_{\mathcal{Z}}\overline{\cG}^-))$.
  Let $\alpha_i=\pi^*(\alpha^+_i,\alpha^-_i)\in H^*(\bigwedge \cG)$, $i=1,2,\cdots,m$. Then for
$$
         \Gamma=\{A^+,g^+,({\bf g}^+),({\bf h}^+),T_k^+;A^-,g^-,({\bf g}^-),({\bf h}^-),T_k^-\},
$$
we have
\begin{eqnarray*}
    \Psi_\Gamma (\alpha_1,\ldots,\alpha_m)&= & \sum_I C(\Gamma,I)\Psi^{(\overline{\cG}^+,\mathcal{Z})}_{(A^+,g^+,(\mathbf{g}^+),({\bf h}^+),T_k^+)}(\alpha^+_{i_1},\ldots,\alpha^+_{i_{m^+}}|\beta^I)\\
       &  & \times\Psi^{(\overline{\cG}^-,\mathcal{Z})}_{(A^-,g^-,(\mathbf{g}^-),({\bf h}^-),T_k^-)}(\alpha^-_{j_1},\ldots,\alpha^-_{j_{m^-}}|\beta_I^*), \nonumber
\end{eqnarray*}
where $\{i_1, \cdots, i_{m^+}\}\cup \{j_1,\cdots,j_{m^-}\} = \{1, \cdots, m\}$, $\beta^I$ runs over all the tuples $(\beta_1,\ldots,\beta_k)$. $\beta_i$ is a basis of $H^*(\mathcal{Z}_{(h_i)})$, and $\beta_I^*$ denotes the dual basis of $\beta^I$,
$C(\Gamma,I):=|Aut(\mathcal{T}(\Gamma,b^I))|\prod^k_{i=1}\ell_i$, where $\mathcal{T}(\Gamma,b^I)=\{(\ell_1,(h_1),\beta_1),\ldots,(\ell_k,(h_k),\beta_k)\}$
\end{theorem}

\begin{remark}\label{remark}
  \begin{enumerate}
    \item In the rest of the paper, we only need this special case that $u^\pm$ has at most one component.
    For the degeneration formula of general $\Gamma$, see Theorem 6.2 in \cite{CLSZ}.
    \item For the case of blow-up at a point, we can show that there is no vanishing 2-cycle (cf. Lemma 2.11 \cite{LR}). Then we have
$$
      \Psi^{\mathcal{G}}_{(A,g,m,(\mathbf{g}))}=\Psi^{\mathcal{G}}_{([A],g,m,(\mathbf{g}))}=\sum \Psi_\Gamma.
$$
   \end{enumerate}
\end{remark}

\begin{remark}
   Fix an index $\Gamma=\{A^+,g^+,({\bf g}^+),({\bf h}^+),T_k^+;A^-,g^-,({\bf g}^-),({\bf h}^-),T_k^-\}$, Set $\Gamma^\pm = \{A^\pm. g^\pm, ({\bf g}^\pm), ({\bf h}^\pm), T_k^\pm\}$.
Denote by $\overline{\mathcal M}_\Gamma$ the component of  $\overline{\mathcal{M}}_{g,({\bf g}),\pi_*[A]}(\overline{\cG}^+\wedge_{\mathcal{Z}}\overline{\cG}^-)$
corresponding to $\Gamma$, then from Proposition \ref{abdim}, we have
$$
    vdim \overline{\mathcal M}_\Gamma = c_1(A) + (n-3)(1-g) + m- \iota^\cG_{({\bf g})}.
$$
Denote by $\overline{\mathcal M}_{\Gamma^\pm}$ the corresponding moduli spaces $\overline{\mathcal M}_{g^\pm, m^\pm, ({\bf g}^\pm), A^\pm,({\bf h}^\pm), T_k^\pm}(\overline{\cG}^\pm, \mathcal{Z})$ respectively.
From the degeneration formula, we know that $\Psi_\Gamma$ is nonzero unless
$$
     vdim\overline{\mathcal M}_\Gamma = vdim\overline{\mathcal M}_{\Gamma^+} + vdim\overline{\mathcal M}_{\Gamma^-} - \dim\prod_j\mathcal{Z}_{(h_j)}.
$$
It follows from Lemma \ref{dim} that $\Psi_\Gamma$ in the degeneration formula is nonzero only if
\begin{eqnarray}\label{dim-sum}
  & & vdim\overline{\mathcal M}_{\Gamma^+} + vdim\overline{\mathcal M}_{\Gamma^-} \\
   & =& \sum_{i=1}^k(n-1-\iota^\mathcal{Z}_{(h_i)}-\iota^\mathcal{Z}_{(h_i^{-1})}) + c_1(A) + (n-3)(1-g) + m-\iota^\cG_{({\bf g})}.\nonumber
\end{eqnarray}
\end{remark}
This formula is a generalization of (5.1) in \cite{LR} to the orbifold category.

\section{Proof of main theorems}

In this section, we will prove our weighted blow-up formulae of orbifold Gromov-Witten invariants at a smooth point. The core of our proof is the dimension
counting of the moduli spaces. Note that we do have a formula (\ref{redim-1}) for the expected dimension of the moduli space $\overline{\mathcal M}_{g,({\bf g}),A,({\bf h}), T_k}(\cG, \mathcal{Z})$.
Since the formula (\ref{redim-1}) contains an undesirable notation [ ], it will make the computation unhandy. Before we prove our main results,
we want to modify the formula (\ref{redim-1}) to make it more easy to use.

After checking the orbifold structure a little more, we have

\begin{proposition}\label{relation}
  The moduli space $\overline{\mathcal{M}}_{g,(\mathbf{g}),A,(\mathbf{h}),T_k}(\mathcal{G},\mathcal{Z})$ is not empty only if for any $ i$
$$
        d_i\equiv r_i(\iota_{(h_i)}^\mathcal{G}-\iota_{(h_i)}^\mathcal{Z})    \quad mod (r_i),
$$
where $d_i$ is the contact order, $r_i$ is the multiplicity of the i-th mark point.
\end{proposition}

\begin{proof}
  Suppose $(\mathcal{C},\mathbf{f})\in \mathcal{M}_{g,(\mathbf{g}),A,(\mathbf{h}),T_k}(\mathcal{G},\mathcal{Z}) $,
  then near relative mark point $x_i$, we have a uniformizing chart $(\mathbb{D},\mathbb{Z}_{r_i},\Phi)$,
  near $z=f(x_i)$, we have a uniformizing chart$(V\times \mathbb{C},G_z,\Psi)$,so we can express $f=(f_0,f_1)$ as (cf. p21 of \cite{CLSZ}):\\
$$
       f:\mathbb{D}\rightarrow V\times \mathbb{C},\quad f_0(w)=(\overline{f}(w),w^{d_i}+O(w^{d_i+1})).
$$
  Suppose $\lambda(f):\mathbb{Z}_{r_i}\rightarrow G_z$ is the homomorphism induced by $f_1$ and
$$
         \lambda(f)(\sigma)=h_i,
$$
where $\sigma$ is the generator of $\mathbb{Z}_{r_i}$.

  By lemma \ref{lemma}, without loss of generality, we may assume $h_i$ is of the form:
\begin{equation*}
    \left(\begin{array}{cc}
            * & 0\\
            0 & a(h_i)
            \end{array}\right).
\end{equation*}

Now, by the definition of degree shifting number, we have
$$
      a(h_i)=e^{2\pi (\iota_{(h_i)}^\mathcal{G}-\iota_{(h_i)}^\mathcal{Z})\sqrt{-1}}.
$$

  Since $f_0(gw)=\lambda(f)(g)\circ f_0(w)$ for any $ w\in \mathbb{D}, g\in \mathbb{Z}_{r_i}$, therefore, for given $g=\sigma=e^{\frac{2\pi \sqrt{-1}}{r_i}}$, we have(just consider the fiber component):$$
         e^{\frac{2\pi d_i\sqrt{-1}}{r_i}}w^{d_i}=(e^{\frac{2\pi \sqrt{-1}}{r_i}}w)^{d_i}=a(h_i)w^{d_i}=e^{2\pi (\iota_{(h_i)}^\mathcal{G}-\iota_{(h_i)}^\mathcal{Z})\sqrt{-1}}w^{d_i}.
$$
  Then we get:
$$
          \frac{d_i}{r_i}-(\iota_{(h_i)}^\mathcal{G}-\iota_{(h_i)}^\mathcal{Z})\in \mathbb{Z},
$$
i.e. $d_i\equiv r_i(\iota_{(h_i)}^\mathcal{G}-\iota_{(h_i)}^\mathcal{Z})\quad mod (r_i)$.
\end{proof}

From Proposition \ref{relation} and Proposition \ref{redim}, it is easy to get another natural formula without the notation [ ].

\begin{corollary}\label{re}
  If the moduli space is nonempty, then
\begin{equation}\label{redim-2}
 vdim\overline{\mathcal{M}}_{g,({\bf g}),A,({\bf h}),T_k}(\mathcal{G},\mathcal{Z})=c_1(A)+(3-n)(g-1)+m+k-\iota_{(\mathbf{g})}^\mathcal{G}-\iota_{(\mathbf{h})}^\mathcal{Z}-\sum_i\ell_i.
\end{equation}
\end{corollary}

\begin{proof} If $d_i\equiv r_i(\iota_{(h_i)}^\mathcal{G}-\iota_{(h_i)}^\mathcal{Z})\quad mod (r_i)$,then we have:
 $$\Sigma[\ell_i]=\Sigma[d_i/r_i]=\Sigma\ell_i-\sum(\iota_{(h_i)}^\mathcal{G}-\iota_{(h_i)}^\mathcal{Z})=\Sigma\ell_i-(\iota_{(\mathbf{h})}^\mathcal{G}-\iota_{(\mathbf{h})}^\mathcal{Z})$$
 Plugging it into the formula (\ref{redim-1}) in Proposition \ref{redim}, we get the formula (\ref{redim-2}).
\end{proof}

 In this paper, We will only consider the case of weighted blow-up at a smooth point.  First of all, we need to fix some notations. Suppose that the weight $\mathbf{m}=(m_1, m_2, \ldots,m_n)$ . Let $\cG$ be a compact symplectic orbifold of dimension $2n$
 and  $P_0$ the blown-up point, We perform the $\mathbf{m}$-weighted symplectic cutting for $\mathcal{G}$ at $P_0$ as in Sect.2.4. We have
$$
        \overline{\mathcal{G}}^+=W\mathbb{P}(m_1,\ldots,m_n,1),\quad \overline{\mathcal{G}}^-=\widetilde{\mathcal{G}}.
$$
Note that the common divisor $\mathcal{Z}\cong W\mathbb{P}(m_1,\ldots,m_n)$ is the exceptional divisor in $\widetilde{\cG}$ and the infinity hyperplane in $W\mathbb{P}(m_1,\cdots,m_n,1)$ respectively.

  Since the first Chern class of weighted projective space plays an important role in the dimension counting, therefore, we need to compute the first Chern class of weighted projective space.
For this purpose, we need the Euler's sequence of weighted projective space as follows( see also Lemma 3.21 of \cite{M})

\begin{proposition}\label{eulerseq}
   (\cite{M}) Suppose $W\mathbb{P}(\mathbf{m})=W\mathbb{P}(m_1,\ldots,m_n)$. Then there exist an exact sequence given by
$$
      0\rightarrow\underline{\mathbb{C}}\xrightarrow[]{\varsigma}\bigoplus^n_{i=1}\mathcal{O}_\mathbf{m}(m_i)\rightarrow TW\mathbb{P}(\mathbf{m})\rightarrow0,
$$
where $\underline{\mathbb{C}}=W\mathbb{P}(\mathbf{m})\times\mathbb{C}$. The map $\varsigma$ is given by $\varsigma(1)=(m_1z_1,\ldots,m_nz_n)$.
\end{proposition}

From this Euler's sequence, we have

\begin{lemma}\label{chernlemma}
Suppose that $E\in H_{2n-2}(W\mathbb{P}(m_1,\cdots,m_n,1))$ is the homology class represented by the divisor $\mathcal{Z}$ as above and $\eta_E$ is its Poincare dual. Then
\begin{equation}\label{chernclass}
     c_1(TW\mathbb{P}(m_1,\cdots,m_n,1)) = (\sum_{i=1}^nm_i +1)\eta_E.
\end{equation}
\end{lemma}

\begin{proof}
 For simplicity, denote $W\mathbb{P}(\mathbf{m}):=W\mathbb{P}(m_1,\ldots,m_n)$, $W\mathbb{P}(\mathbf{m}'):=W\mathbb{P}(m_1,\ldots,m_n,1)$.
 Consider the orbifold embedding induced by $\mathcal{Z}$:
 \begin{align*}
   s:\quad W\mathbb{P}(\mathbf{m})\quad\quad &\longrightarrow  W\mathbb{P}(\mathbf{m}')\\
     [(z_1,\ldots, z_n)] & \longmapsto [(z_1,\ldots,z_n,0)]
 \end{align*}
 which induce a map of cohomology groups:
$$
   s^*:\quad H^*(W\mathbb{P}(\mathbf{m}'))\quad\longrightarrow \quad H^*(W\mathbb{P}(\mathbf{m})),
$$
 and we have:
$$
   s^{-1}(TW\mathbb{P}(\mathbf{m}'))=TW\mathbb{P}(\mathbf{m}')|_{\mathcal{Z}}=TW\mathbb{P}(\mathbf{m})\bigoplus \mathcal{N},
$$
 where $\mathcal{N}=\mathcal{N}_{\mathcal{Z}|W\mathbb{P}(\mathbf{m}')}=\mathcal{O}_\mathbf{m}(1)$ (cf. Proposition \ref{norbund}). So from Proposition \ref{eulerseq} and Whitney sum formula, we have:
\begin{equation}\label{E1}
       s^*c_1(TW\mathbb{P}(\mathbf{m}'))=\sum^n_{i=1}c_1(\mathcal{O}_\mathbf{m}(m_i))+c_1(\mathcal{O}_\mathbf{m}(1))=(\sum^n_{i=1}m_i+1)e(\mathcal{O}_\mathbf{m}(1)).
\end{equation}
 Denote $\Phi(\mathcal{N})$ is the Thom class of orbi-bundle $\mathcal{N}$, by the well-known relation among Euler class, Thom class and Poincare dual, we have:
\begin{equation}\label{E2}
      s^*(\eta_E)=s^*(\Phi(\mathcal{N}))=s^*(\Phi(\mathcal{O}_\mathbf{m}(1)))=e(\mathcal{O}_\mathbf{m}(1)).
\end{equation}
Combining (\ref{E1}) and (\ref{E2}) we get:
\begin{equation}\label{E3}
           c_1(TW\mathbb{P}(\mathbf{m}'))-(\sum^n_{i=1}m_i+1)\eta_E\in Ker(s^*).
\end{equation}
 Now we consider the long exact sequence:
$$
     \cdots\rightarrow H^2(W\mathbb{P}(\mathbf{m}'),W\mathbb{P}(\mathbf{m}))\rightarrow H^2(W\mathbb{P}(\mathbf{m}'))
  \xrightarrow[]{s^*} H^2(W\mathbb{P}(\mathbf{m}))\rightarrow H^3(W\mathbb{P}(\mathbf{m}'),W\mathbb{P}(\mathbf{m}))\rightarrow\cdots
$$
Note that:
\begin{align*}
  H^2(W\mathbb{P}(\mathbf{m}'),W\mathbb{P}(\mathbf{m}))&= H^2(W\mathbb{P}(\mathbf{m}'),U)\xlongequal[]{exision}
  H^2(W\mathbb{P}(\mathbf{m}')-\mathcal{Z},U-\mathcal{Z})\\
  &= H^2(\mathbb{C}^n,\mathbb{C}^n-\{0\})=0,
\end{align*}
 where $U$ is a small neighborhood near $\mathcal{Z}$ which is homotopic to $\mathcal{Z}$.

So $s^*: H^2(W\mathbb{P}(\mathbf{m}'))\longrightarrow H^2(W\mathbb{P}(\mathbf{m}))$ is an injective. (In fact, it is an isomorphism.) This fact together with (\ref{E3}) imply that:
$$
        c_1(TW\mathbb{P}(\mathbf{m}'))=(\sum^n_{i=1}m_i+1)c_1(\eta_E).
$$
\end{proof}

Next we will follow \cite{Hu} to decompose the proof of our main theorems into two comparison theorems between absolute and relative Gromov-Witten invariants. The first coming theorem is
\begin{theorem}\label{comthm1}
    Under the assumption of Theorem \ref{thm1-1}. If $g\leq 1$, $n\geq 2$, then
\begin{equation}\label{comp1}
       \Psi^{\mathcal{G}}_{(A,g,m,(\mathbf{g}))}(\alpha_1,\ldots,\alpha_m)=\Psi^{(\overline{\mathcal{G}}^-,\mathcal{Z})}_{(A^-,g,m,(\mathbf{g}))}(\alpha_1,\ldots,\alpha_m)
\end{equation}
\end{theorem}

\begin{proof}
 We perform the ${\bf m}$-weighted symplectic cutting for $\cG$ at $P_0$. Then we get
$$
    \overline{\cG}^+= W\mathbb{P}(m_1,\cdots, m_n,1), \quad \overline{\cG}^- = \widetilde{\cG}.
$$
 Now we want to apply the degeneration formula to compute the absolute GW invariants of $\cG$ in the LHS of (\ref{comp1}). From Remark \ref{remark} (2) and the degeneration formula,
 we only need to consider the contribution of each component $\Gamma=\{\Gamma^+,\Gamma^-\}$ to the GW-invariants.
 According to our convention, $u^\pm:\mathcal{C}^\pm\longrightarrow\mathcal{G}^\pm$ may have many connected components $u^\pm_i:\mathcal{C}^\pm_i\longrightarrow\mathcal{G}^\pm,i=1,\ldots,l^\pm$.
 Suppose $\mathcal{C}^\pm_i$ has genus $g^\pm_i$, $g^\pm=\sum g^\pm_i$ with $m^\pm_i$ marked points. Note that
 $\overline{\mathcal{G}}^+=W\mathbb{P}(m_1,\ldots,m_n,1)$. For the index $\Gamma^+$, from (\ref{redim-2}), we have
 \begin{equation}\label{ind++}
   vdim\overline{\mathcal M}_{\Gamma^+} =\sum c_1([u^+_i])+(3-n)(g^+-l^+)+m^++k-\iota_{(\mathbf{g}^+)}^\mathcal{G}-\iota_{(\mathbf{h})}^\mathcal{Z}-\sum \ell_i,
 \end{equation}
where the last summation runs over all fractional contact orders in $\Gamma^+$.

 Since $\mathcal{Z}\cong W\mathbb{P}(m_1,\ldots,m_n)$ in $W\mathbb{P}(m_1,\ldots,m_n,1))$,
  an intersection multiplicity calculation shows $\sum([u^+_i]\cdot \eta_E)=\sum\ell_i$.

Therefore, from (\ref{chernclass}), we have
\begin{equation}\label{chernnumber-1}
     \sum c_1([u^+_i])=(\sum^n_{i=1}m_i+1)(\sum^k_i\ell_i).
\end{equation}

Combining (\ref{ind++}) and (\ref{chernnumber-1}), we have
\begin{equation}\label{ind+}
 vdim\overline{\mathcal M}_{\Gamma^+} =(\sum^n_{i=1}m_i)(\sum^k_{j=1} \ell_j)+(3-n)(g^+-l^+)+m^++k-\iota_{(\mathbf{g}^+)}^\mathcal{G}-\iota_{(\mathbf{h})}^\mathcal{Z}.
\end{equation}

Since $\alpha_i\in H^*(\mathcal{G})$ and the blown-up point $P_0$ is a smooth point, we may assume that all $\alpha_i$ support away from the neighborhood of $P_0$.
So we have $\alpha_i^+=0$, $1\leq i\leq m$. Therefore, if $m^+>0$, we have for any $\beta_b\in H^*(\mathcal{Z})$,
$$
        \Psi^{(\overline{\mathcal{G}}^+,\mathcal{Z})}_{(A^+,g^+,m^+,(\mathbf{g}^+),\{\ell_1,\ldots,\ell_k\},(\mathbf{h}))}(\alpha^+_i,\beta_b)=0.
$$
This implies $\Psi_\Gamma=0$ except $m^-=m$.

Now we assume $m^-=m$, i.e. $m^+=0$. From (\ref{dim-sum}) and (\ref{ind+}), we get:
\begin{eqnarray}
     vdim \overline{\mathcal M}_{\Gamma^-} = & & (n-2)k-\sum^k_{i=1}\iota_{(h_i^{-1})}^\mathcal{Z}+c_1(A)+(3-n)(g-g^++l^+-1) \nonumber  \\
     & & +\,\, m-\iota_{(\mathbf{g})}^\mathcal{G}-(\sum^n_{i=1}m_i)(\sum^k_{j=1} \ell_j),  \label{ind-}
\end{eqnarray}
where $k$ is the number of relative marked points on $\mathcal{C}^+$ or $\mathcal{C}^-$.

On the other hand, if
$$
      \frac{1}{2}\sum \deg\alpha_i\neq c_1(A)+(3-n)(g-1)+m-\iota_{(\mathbf{g})}^\mathcal{G},
$$
by the definition of the orbifold Gromov-Witten invariant, we have
$$
      \Psi^\mathcal{G}_{(A,g,m,(\mathbf{g}))}(\alpha_1,\ldots,\alpha_m)=\Psi^{\widetilde{\mathcal{G}}}_{(p!(A),g,m,(\mathbf{g}))}(p^*\alpha_1,\ldots,p^*\alpha_m)=0.
$$
Then the theorem holds trivially. Therefore, we may also assume
\begin{equation}\label{deg}
  \frac{1}{2}\sum \deg\alpha_i=c_1(A)+(3-n)(g-1)+m- \iota_{(\mathbf{g})}^\mathcal{G}.
\end{equation}
(\ref{ind-}) and (\ref{deg}) imply that:
\begin{equation}\label{deg-ind}
  \frac{1}{2}\sum deg\alpha_i- vdim\overline{\mathcal M}_{\Gamma^-} =(\sum^n_{i=1}m_i)(\sum^k_{j=1} \ell_j)+(3-n)(g^+-l^+)+\sum^k_{j=1}\iota_{(h_j^{-1})}^\mathcal{Z}-(n-2)k.
\end{equation}

Next, we first prove the following lemma:

\begin{lemma}\label{noneq}
  If $l^+\geq 1$, $k>0$, $g\leq 1$, $n\geq 2$, then
$$
          \frac{1}{2}\sum deg\alpha_i- vdim\overline{\mathcal M}_{\Gamma^-}>0.
$$
\end{lemma}

\begin{proof}
  Since $k\geq l^+\geq 1$, $g^+\leq g\leq 1$, $n\geq 2$, it is easy to see that
$$
         (3-n)(g^+-l^+)+k\geq 0.
$$
From (\ref{deg-ind}), it suffices to prove that
$$
      (\sum^n_{i=1}m_i)(\sum^k_{j=1} \ell_j)+\sum^k_{j=1}\iota_{(h_j^{-1})}^\mathcal{Z}-(n-1)k>0.
$$

   In fact, we will prove that, for $1\leq j\leq k$
\begin{equation}\label{neq}
     (\sum^n_{i=1}m_i)\ell_j+\iota_{(h_j^{-1})}^\mathcal{Z}-(n-1)>0.
\end{equation}
The proof of (\ref{neq}) is nothing but direct checking. For simplicity, we drop the index $j$. The orbifold J-holomorphic map
 $u^+:\mathcal{C}\longrightarrow \overline{\mathcal{G}}^+\cong W\mathbb{P}(m_1,\ldots,m_n,1)$ maps the relative marked point $x$
 to $(p,(h))\in\bigwedge\mathcal{Z}$ with fractional contact order $\ell$, where $\bigwedge\mathcal{Z}$ is the inertia orbifold of $\mathcal{Z}$.
 Then $\ell=\frac{d}{r}$ , where $d$ is the contact order of the lifting map and $r$ is the order of $h$.

Suppose $p=[(z_1,\ldots,z_n,w)]\in W\mathbb{P}(m_1,\ldots,m_n,1)$, since $p\in\mathcal{Z}$, there is some $i$ s.t. $z_i\neq 0$.
Without loss of generality, we assume $z_1\neq 0$, then $p$ is on the standard orbifold chart $U_1=\{[(z_1,\ldots,z_n,w)]|z_1\neq 0\}$ of $W\mathbb{P}(m_1,\ldots,m_n,1)$:
(c.f. Sect. 2.3)
\begin{align*}
  \varphi: & \qquad \quad U_1\qquad \longrightarrow  \qquad\mathbb{C}^n/\mathbb{Z}_{m_1}\\
           &[(z_1,\ldots,z_n,w)]  \longmapsto   (\frac{z_2}{z_1^{\frac{m_2}{m_1}}},\ldots,\frac{z_n}{z_1^{\frac{m_n}{m_1}}},\frac{w}{z_1^{\frac{1}{m_1}}})_{m_1},
\end{align*}
where $\xi\in\mathbb{Z}_{m_1}$ acts on $\mathbb{C}^n$ as: $\xi\cdot(y_1,\ldots, y_n)=(\xi^{m_2}y_1,\ldots,\xi^{m_n}y_{n-1},\xi y_n)$.
Now we compute the degree shifting number $\iota_{(h^{-1})}^\mathcal{Z}$. Suppose the map $\lambda_{u^+}:G_{x}\rightarrow \mathbb{Z}_{m_1}$ which is induced by $u^+$
maps the generator of $G_{x}$ to $\xi=e^{\frac{2l\pi \imath}{m_1}}\in\mathbb{Z}_{m_1}, \quad(0\leq l\leq m_1-1)$. Then
$$\xi^{-1}\cdot(y_1,\ldots, y_n)=(e^{-\frac{2m_2l\pi \imath}{m_1}}y_1,\ldots,e^{-\frac{2m_nl\pi \imath}{m_1}}y_{n-1},e^{-\frac{2l\pi \imath}{m_1}} y_n)$$
And the orbifold chart of $\mathcal{Z}$ is just the projection of the first n-1 components of $\mathbb{C}^n$. By the definition of degree shifting number, we have:
\begin{align}\label{eq1}
  \iota_{(h^{-1})}^\mathcal{Z}=\sum_{i=2}^n\frac{l_im_1-lm_i}{m_1}
\end{align}
where $l_i$ is minimal integer s.t. $\frac{l_im_1-lm_i}{m_1}\geq 0$, $l_i\geq 0$.

On the other hand, it is easily to see that $ \iota_{(h)}^\mathcal{G}-\iota_{(h)}^\mathcal{Z}=\frac{l}{m_1}$. Recall that in the proof of
corollary \ref{re}, we have
$$q:=\ell-(\iota_{(h)}^\mathcal{G}-\iota_{(h)}^\mathcal{Z})=[\ell]\in\mathbb{N}$$
So
\begin{align}\label{eq2}
  \ell=q+(\iota_{(h)}^\mathcal{G}-\iota_{(h)}^\mathcal{Z})=\frac{qm_1+l}{m_1}
\end{align}
Put (\ref{eq1}) and (\ref{eq2}) into the left-handside of (\ref{neq}), we denote $l_1:=l$ and get
\begin{eqnarray*}
 (\sum^n_{i=1}m_i)\ell+\iota_{(h^{-1})}^\mathcal{Z}-(n-1)&=&(\sum^n_{i=1}m_i)\frac{qm_1+l}{m_1}+\sum_{i=2}^n\frac{l_im_1-lm_i}{m_1}-(n-1)\\
                                                         &=&(qm_1+l)+\sum_{i=2}^n\frac{qm_1m_i+lm_i+l_im_1-lm_i}{m_1}-(n-1)\\
                                                         &=&(\sum^n_{i=1}m_i)q+\sum_{i=1}^nl_i-(n-1)\\
                                                         &=&\sum^n_{i=1}(m_iq+l_i-1)+1.
\end{eqnarray*}
Note that $l_i,q\in\mathbb{N}$, and :\\
(1)If $q\geq 1$, then it is easily see that $m_iq+l_i-1 \geq 0$, (\ref{neq}) holds;\\
(2)If $q=0,l\geq 1$, then from (\ref{eq1}), we have $l_i\geq 1,\quad \forall i$, then $m_iq+l_i-1 \geq 0$, (\ref{neq}) holds;\\
(3)If $q=0,l=0$, then from (\ref{eq1}), we have $l_i=0,\, \forall i$, then the relative mark point $x$ is a smooth point of contact order $0$,
   contradicting to the definition of relative mark point.

  Summarizing, we complete the proof of the lemma.
\end{proof}

Now we come back to the proof of the comparison theorem.  If $k>0$, by the definition of relative Gromov-Witten invariant and Lemma \ref{noneq},
we have for any $\beta_b\in H^*_{CR}(\mathcal{Z},\mathbb{C})$,
$$
         \Psi^{(\overline{\mathcal{G}}^-,\mathcal{Z})}_{(A^-,g^-,m,(\mathbf{g}),\{\ell_1,\ldots,\ell_k\},(\mathbf{h}))}(\alpha^-_i,\beta_b)=0.
$$
Therefore, $\Psi_\Gamma=0$ except for $\Gamma=\{A^-,g,m,(\mathbf{g})\}$. This completes the proof of the comparison theorem.
\end{proof}

\begin{remark}\label{remark2}
    In fact, Lemma \ref{noneq} still holds when $g\geq 2$, $n$=2 or 3. This fact can be easily seen in the proof.
\end{remark}

Next,  we consider the case of $\widetilde{\mathcal{G}}$ and prove our second comparison theorem.

\begin{theorem}\label{comthm2}
  Under the assumption of Theorem \ref{thm1-1}. If $g\leq 1$, $n\geq 2$, then
\begin{align}\label{comp2}
       \Psi^{\widetilde{\mathcal{G}}}_{(p^!(A),g,m,(\mathbf{g}))}(p^*(\alpha_1),\ldots,p^*(\alpha_m))=\Psi^{(\overline{\mathcal{G}}^-,\mathcal{Z})}_{(p^!(A)^-,g,m,(\mathbf{g}))}(\alpha_1,\ldots,\alpha_m)
\end{align}
\end{theorem}

\begin{proof}
  We perform the orbifold symplectic cutting with trivial weight along the exceptional divisor $E$. We have
$$
       \overline{\widetilde{\mathcal{G}}}^+=\mathbb{P}(\mathcal{O}_{(m_1,\ldots,m_n)}(-1)\oplus\mathcal{O}),\quad \overline{\widetilde{\mathcal{G}}}^-=\widetilde{\mathcal{G}}.
$$
   Now we apply the gluing theorem to compute the contribution of each gluing component. In fact, we will prove that the contribution of relative stable J-holomorphic curves in
   $\widetilde{\mathcal{G}}$ which touch the exceptional divisor $E$ to the GW-invariant of $\widetilde{\mathcal{G}}$ is zero. We consider the component
$$
       \Gamma=\{p^!(A)^+,g^+,m^+,(\mathbf{g}^+),\{\ell_1,\ldots,\ell_k\},(\mathbf{h}^+);p^!(A)^-,g^-,m^-,(\mathbf{g}^-),\{\ell_1,\ldots,\ell_k\},(\mathbf{h^-})\},
$$
with $({\bf h}^+) = (({\bf h}^-)^{-1})$.

Denote $\Gamma^\pm =\{p^!(A)^\pm, g^\pm,m^\pm,({\bf g}^\pm), \{\ell_1,\cdots,\ell_k\}, ({\bf h}^\pm)$. Then from (\ref{dim-sum}) we have:
\begin{eqnarray}\label{ind+ind2}
    & & vdim\overline{\mathcal M}_{\Gamma^+} + vdim\overline{\mathcal M}_{\Gamma^-}  \nonumber\\
    & = &  \sum^k_{i=1}(n-1-\iota_{(h_i)}^\mathcal{Z}-\iota_{(h_i^{-1})}^\mathcal{Z})+c_1(p^!(A))+(3-n)(g-1)+m-\iota_{(\mathbf{g})}^{\widetilde{\mathcal{G}}}.
\end{eqnarray}

   Now we want to calculate $vdim\overline{\mathcal M}_{\Gamma^+}$. As in the proof of Theorem \ref{comthm1}, we assume $u^{\pm}: \Sigma^{\pm}\rightarrow\mathcal{G}^{\pm}$ has $l^{\pm}$ connected components $u^{\pm}_i: \Sigma^{\pm}_i\rightarrow\mathcal{G}^{\pm}$, $i=1,\ldots,l^{\pm}$
   and $\Sigma^{\pm}_i$ has antiemetic genus $g^{\pm}_i$, $g^{\pm}=\Sigma g^{\pm}_i$ with $m^{\pm}_i$ marked points. Then we have the same formula of $vdim\overline{\mathcal M}_{\Gamma^+}$ as (\ref{ind++}).
  Now we calculate $\sum c_1([u^+_i])$ in (\ref{ind++}) as follows:

   Observe that we obtain $\overline{\widetilde{\mathcal{G}}}^+$ by performing the symplectic cutting twice.
   We also note that $\overline{\widetilde{\mathcal{G}}}^+$ is independent of the order of these two orbifold symplectic cuttings. Therefore, if we commute the order of
   these two symplectic cuttings, it is easy to see:
   $$\mathbb{P}(\mathcal{O}_{(m_1,\ldots,m_n)}(-1)\oplus\mathcal{O})\cong\widetilde{W\mathbb{P}}(m_1,\ldots,m_n,1)$$
   where $\widetilde{W\mathbb{P}}(m_1,\ldots,m_n,1)$ is obtained by performing $(m_1,\ldots,m_n)$-weighted blow-up on the smooth point $[(0,\ldots,0,1)]$ of $W\mathbb{P}(m_1,\ldots,m_n,1)$.

   Let $E$ denotes the zero section of $\mathbb{P}(\mathcal{O}_{(m_1,\ldots,m_n)}(-1)\oplus\mathcal{O})$, and $H$ denotes the infinity section. Then we have:
   $$E\cong H\cong W\mathbb{P}(m_1,\ldots,m_n)$$

   Now we need to consider the cohomology group of $\widetilde{W\mathbb{P}}(m_1,\ldots,m_n,1)$.
   Godinho \cite{G} gave an interpretation of the underlying space of blow-up in terms of connected sum, see Lemma 5.1 of \cite{G}. The following Proposition is just the special case of $p=1$.

\begin{proposition}\label{connectsum}
     Suppose $\cG$ is a symplectic orbifold, the weighted ${\bf m}$-blow-up of $\cG$ at a smooth point $x$ is given by the connected sum:
$$
            \widetilde{\cG}=\cG \# W\mathbb{P}({\bf m},1).
$$
\end{proposition}

From this Proposition and (\ref{coh}), we have:
$$
         H^2(\widetilde{W\mathbb{P}}(m_1,\ldots,m_n,1))=\R\oplus\R
$$
   and the generator is easily to seen to be $\eta_E$ and $\eta_H$, where $\eta_E$ and $\eta_H$ is the Poincare dual of $E$ and $H$.
   Now we have:
   \begin{eqnarray*}
     (\sum c_1([u^+_i]))\cdot H & = & p^!(A)^+\cdot H=\sum\ell_i,\\
     (\sum c_1([u^+_i]))\cdot E & = & p^!(A)^+\cdot E=p^!(A)\cdot E=A\cdot p_*(E)=0.
   \end{eqnarray*}
   And we have an analog of Lemma \ref{chernlemma}:

\begin{lemma}\label{chernlemma-1}
\begin{equation}\label{chernclass-1}
      c_1(T\widetilde{W\mathbb{P}}(m_1,\ldots,m_n,1))=(\sum^n_{i=1}m_i+1)\eta_H-(\sum^n_{i=1}m_i-1)\eta_E.
\end{equation}
\end{lemma}

\begin{proof}
  Similar to the proof of Lemma \ref{chernlemma}, denote $W\mathbb{P}(\mathbf{m}):=W\mathbb{P}(m_1,\ldots,m_n)$,
   $\widetilde{W\mathbb{P}}(\mathbf{m}'):=\\ \widetilde{W\mathbb{P}}(m_1,\ldots,m_n,1)$.
   Suppose $c_1(\widetilde{W\mathbb{P}}(\mathbf{m}'))=a\eta_H+b\eta_E$.
 Consider the orbifold embedding induced by $E$:
$$
   s:\quad W\mathbb{P}(\mathbf{m})\quad\hookrightarrow \quad \widetilde{W\mathbb{P}}(\mathbf{m}')
$$
 which induces a map of cohomology groups:
$$
   s^*:\quad H^*(\widetilde{W\mathbb{P}}(\mathbf{m}'))\quad\longrightarrow \quad H^*(W\mathbb{P}(\mathbf{m})).
$$
 And we have:
$$
   s^{-1}(T\widetilde{W\mathbb{P}}(\mathbf{m}'))=T\widetilde{W\mathbb{P}}(\mathbf{m}')|_{E}=TW\mathbb{P}(\mathbf{m})\bigoplus \mathcal{N},
$$
 where $\mathcal{N}=\mathcal{N}_{E|\widetilde{W\mathbb{P}}(\mathbf{m}')}=\mathcal{O}_\mathbf{m}(-1)$. So via Proposition \ref{connectsum} and Whitney sum formula, we have:
\begin{equation}\label{T}
         s^*c_1(T\widetilde{W\mathbb{P}}(\mathbf{m}'))=\sum^n_{i=1}c_1(\mathcal{O}_\mathbf{m}(m_i))+c_1(\mathcal{O}_\mathbf{m}(-1))=(\sum^n_{i=1}m_i-1)e(\mathcal{O}_\mathbf{m}(1)).
\end{equation}
 Denote $\Phi(\mathcal{N})$ is the Thom class of orbi-bundle $\mathcal{N}$, by the well-known relation among Euler class, Thom class and Poincare dual, we have:
\begin{equation}\label{O(-1)}
         s^*(\eta_E)=s^*(\Phi(\mathcal{N}))=s^*(\Phi(\mathcal{O}_\mathbf{m}(-1)))=e(\mathcal{O}_\mathbf{m}(-1))=-e(\mathcal{O}_\mathbf{m}(1)).
\end{equation}
 Combining (\ref{T}) and (\ref{O(-1)}), we get
$$
         c_1(T\widetilde{W\mathbb{P}}(\mathbf{m}'))+(\sum^n_{i=1}m_i-1)\eta_E\in Ker(s^*).
$$
 Now we consider the long exact sequence
$$
       \cdots\rightarrow H^2(\widetilde{W\mathbb{P}}(\mathbf{m}'),E)\xrightarrow[]{j^*} H^2(W\mathbb{P}(\mathbf{m}'))
       \xrightarrow[]{s^*} H^2(W\mathbb{P}(\mathbf{m}))\rightarrow H^3(\widetilde{W\mathbb{P}}(\mathbf{m}'),E)\rightarrow\cdots
$$
Note that
 \begin{eqnarray*}
  H^2(\widetilde{W\mathbb{P}}(\mathbf{m}'),E)&= &H^2(W\mathbb{P}(\mathbf{m}'),U)\xlongequal[]{exision}
  H^2(\widetilde{W\mathbb{P}}(\mathbf{m}')-E,U-E)\\
  &= &H^2(\mathcal{N}_{H|\widetilde{W\mathbb{P}}(\mathbf{m}')},\mathcal{N}_{H|\widetilde{W\mathbb{P}}(\mathbf{m}')}-H),
 \end{eqnarray*}
 where $U$ is a small neighborhood near $E$ which is homotopic to $E$. Now we can know that $Ker(s^*)=j^*H^2(\widetilde{W\mathbb{P}}(\mathbf{m}'),E)$ is generated by $j^*\Phi(\mathcal{N}_{H|\widetilde{W\mathbb{P}}(\mathbf{m}')})=\eta_H$.
 Then we get:
$$
       b=-(\sum^n_{i=1}m_i-1).
$$
 We can use the same method to get:
$$
           a=\sum^n_{i=1}m_i+1.
$$
 Finally we have:
$$
     c_1(T\widetilde{W\mathbb{P}}(m_1,\ldots,m_n,1))=(\sum^n_{i=1}m_i+1)\eta_H-(\sum^n_{i=1}m_i-1)\eta_E.
$$
This completes the proof of the Lemma.
\end{proof}

Next, we come back to the proof of Theorem \ref{comthm2}. From (\ref{chernclass-1}), we have
$$
       \sum c_1([u^+_i])=(\sum^n_{i=1}m_i+1)(\sum^k_i\ell_i).
$$

The same argument as in proof of Theorem \ref{comthm1} shows that we only need to prove
\begin{eqnarray*}
    \frac{1}{2}\sum deg\alpha_i- vdim\overline{\mathcal M}_{\Gamma^-} & = & (\sum^n_{i=1}m_i)(\sum^k_{j=1} \ell_j)+(3-n)(g^+-l^+)+\sum^k_{j=1}\iota_{(h_j^{-1})}^\mathcal{Z}-(n-2)k \\
          & > &0.
\end{eqnarray*}
The proof of the last inequality is totally the same as the proof of Lemma \ref{noneq}. This completes the proof of the second comparison Theorem.

\end{proof}

Summarizing Theorem \ref{comthm1} and Theorem \ref{comthm2} and note that $A^- = p^!(A)^-$, we get the blow-up formula
\begin{theorem}
   If $g\leq 1$, $n\geq 2$, then
$$
      \Psi^{\mathcal{G}}_{(A,g,m,(\mathbf{g}))}(\alpha_1,\ldots,\alpha_m)=\Psi^{\widetilde{\mathcal{G}}}_{(p^!(A),g,m,(\mathbf{g}))}(p^*(\alpha_1),\ldots,p^*(\alpha_m)).
$$
 \end{theorem}

 From Remark \ref{remark2}, we can also conclude that:
\begin{theorem}\label{main}
   If $dim_{\mathbb{R}} \cG=$4 or 6, then
$$
      \Psi^{\mathcal{G}}_{(A,g,m,(\mathbf{g}))}(\alpha_1,\ldots,\alpha_m)=\Psi^{\widetilde{\mathcal{G}}}_{(p^!(A),g,m,(\mathbf{g}))}(p^*(\alpha_1),\ldots,p^*(\alpha_m)).
$$
 \end{theorem}

 \begin{remark}
   Unfortunately, we do not know wether Theorem \ref{main} holds when $n=1$ or $g\geq 2$, since Lemma \ref{noneq} may fail in that case. We conjecture that the condition $g\leq 1$ is a technical one, but $n\geq 2$ is not. We will study this problem in the future.
 \end{remark}

 \section{Application}

 In this section, we will give an application of our main result Theorem \ref{main}. Uniruledness is an important property in the birational classification of algebraic manifold. In \cite{HLR}, Hu, Li and Ruan introduces the concept of symplectic uniruledness via Gromov-Witten invariant: a symplectic manifold $(M, \omega)$ is symplectic  uniruled if there is a nonzero genus zero GW invariant involving a point class. 

 Note that the symplectic uniruledness is easily to generalize to orbifold, we can obtain a result about uniruledness in the category of orbifold:
 \begin{corollary}
   If the condition of $\cG$ and $\widetilde{\cG}$ are as in Theorem \ref{main}, then
   \begin{center}
   $\cG$ is symplectic uniruled $\Rightarrow$ $\widetilde{\cG}$ is symplectic uniruled.
   \end{center}

 \end{corollary}

 \begin{proof}
    Suppose $\cG$ is symplectic uniruled, then $\exists A \in H_2(|\cG|,\mathbb{Z})$ and $\alpha_2,\ldots,\alpha_m \in H^*(\cG)$, such that  $\Psi^{\mathcal{G}}_{(A,g,m,(\mathbf{g}))}([pt],\alpha_2,\ldots,\alpha_m)\neq 0$. By Theorem \ref{main}, we have $$\Psi^{\widetilde{\mathcal{G}}}_{(p^!(A),g,m,(\mathbf{g}))}([pt],p^*(\alpha_2),\ldots,p^*(\alpha_m))\neq 0$$. So $\widetilde{\cG}$ is symplectic uniruled.
 \end{proof}

\vskip .2in
\noindent
{\bf Acknowledgments} The authors like to thank Bohui Chen for  many invaluable discussions during the course of this work. The first author thanks Department of Mathematics, University of Michigan
and the second author thanks Max-Planck Institute for Mathematics at Bonn for their hospitality during part of writing of this paper.

  \end{document}